\documentclass[12pt]{amsart}
\usepackage{amsaddr}
\usepackage{amsmath, amsthm, amscd, amsfonts, amssymb, graphicx, xcolor,comment, mathtools}
\usepackage[numbers,square]{natbib}
\usepackage[colorlinks,citecolor=black,urlcolor=black]{hyperref}
\usepackage{setspace}
\newcommand{\dd}{\,\mathrm d}

\newtheorem{theorem}{Theorem}[section]
\newtheorem{lemma}[theorem]{Lemma}
\newtheorem{proposition}{Proposition}[section]

\theoremstyle{definition}
\newtheorem{definition}{Definition}[section]
\newtheorem{example}{Example}[section]

\theoremstyle{remark}
\newtheorem{remark}{Remark}[section]
\numberwithin{equation}{section}

\title[Strong Negative Type]{Strong Negative Type of Star-shaped Manifolds}

\author[Hang Liu, Ziyan Xu]{Hang Liu$^1$, Ziyan Xu$^2$}
\address{$^1$ Department of Statistics and Finance, School of Management,
University of Science and Technology of China, Hefei, China.\\
$^2$ School of Mathematical Sciences,
University of Science and Technology of China, Hefei, China.}
\email{hliu01@ustc.edu.cn, ustc-xzy@mail.ustc.edu.cn}

\begin{document}

\begin{abstract}

Metric spaces of negative type (NT) and strong negative type (SNT) are fundamental in wide applications across computer science and statistics. It is well-known that spheres equipped with geodesic distance are of NT, and subsets of a sphere that contain at most one pair of antipodal points (e.g.\ an open hemisphere) are of SNT. In this paper, we extend these results to a class of star-shaped manifolds. By constructing a metric $d$ based on layer-wise radial projections,  
we prove that every star-shaped manifold, equipped with $d$, is of NT. Moreover, we prove that a subset of a star-shaped manifold is of SNT under certain injectivity and directional asymmetry conditions. In particular, for a subset of a single layer, it is of SNT if it contains exactly one point from each preimage of an open hemisphere (a fully directional asymmetry); moreover, it is of SNT if and only if it contains at most one opposite-directional pair. Finally, we prove that under similar conditions, SNT holds for almost and essentially star-shaped manifolds, and for those with incomplete radial coverage.




\noindent \textit{MSC2020 subject classifications:} 53C23, 53C25.
\end{abstract}
\maketitle

\section{Introduction}

A metric space $(\mathcal{X}, d)$ is said to be of {\it negative type} (NT) if, for any $\ell\ge 1$, $x_i\in \mathcal{X}$, and $\alpha_i\in \mathbb{R}$, $i = 1,2,\dots,\ell$ such that~$\sum_{i=1}^\ell \alpha_i = 0$,  
 \[\sum_{1\le i,j\le \ell}\alpha_i\alpha_j d(x_i,x_j) \le 0.\]
 An alternative definition is that $(\mathcal{X}, d^{1/2})$ can be isometrically embedded into a Hilbert space \cite{schoenberg1937, schoenberg1938}.
 Let $\mathcal{P}_1(\mathcal{X})$ denote the set of probability measures on $\mathcal{X}$ with finite first-order moment. We say that~$(\mathcal{X}, d)$ is of {\it strong negative type} (SNT) if (a)  $(\mathcal{X}, d)$ is of NT and (b) for any~$\mu_1, \mu_2 \in \mathcal{P}_1(\mathcal{X})$, the \emph{$d$-energy $E_d$} of $\mu_1-\mu_2$ satisfies
 \begin{equation}\label{eq.SNT}
E_d(\mu_1-\mu_2) \coloneqq \int_{\mathcal{X}\times \mathcal{X}} d(x,y) {\rm d}(\mu_1-\mu_2)^2(x,y)\leq 0,
 \end{equation}
with equality if and only if $\mu_1 = \mu_2$.

Metric spaces of NT have wide applications in computer science and statistics, including graph embedding \cite{naor2010l1, campbell2025}, kernel-based machine learning \cite{feragen2015}, and outlier detection \cite{huntsman2025}. However, for some applications (e.g.\ independence test and homogeneity test), NT alone is insufficient and SNT is required \cite{Szekely2007, lyons2013, szekely2013,Discussion2025}.

It is well-known that spheres equipped with geodesic distance are of NT \cite{meckes2013}, and subsets of a sphere that contain at most one pair of antipodal points (e.g.\ an open hemisphere) are of SNT \cite{lyons2020}. In this paper, we extend these results to a class of star-shaped manifolds. We start from constructing a metric based on
layer-wise radial projections. Then we prove that every star-shaped manifold, equipped with this metric, is of NT.
Moreover, we prove that a subset of a star-shaped manifold is of SNT under certain injectivity and directional asymmetry conditions. In particular, for a subset of a single layer, it is of SNT if it contains exactly one point from each preimage of an open hemisphere (a fully directional asymmetry); moreover, it is of SNT if and only if it contains at most one opposite-directional pair.

We further consider manifolds with incomplete radial coverage,
which naturally arise for non-convex geometries with substantial holes.
This provides an extension of the negative-type construction beyond the complete radial setting. The theoretical results are illustrated on explicit geometries, including the ellipsoid, rotational toroidal hypersurface and product torus.

The rest of this paper is organised as follows. Section~\ref{sec.main} introduces the terminology that is related to star-shaped manifolds, and states the main results for the NT and SNT of star-shaped manifolds and their submanifolds. Three examples---ellipsoid, rotational toridal hypersurface, and product torus---conclude this section. Finally, all proofs are collected in Section~\ref{sec.proof}.

\subsection*{AI Declaration}

AI tool (OpenAI's ChatGPT) was used only to improve the grammar of the manuscript. The authors take
full responsibility for the mathematical results
and final content of the manuscript.





\section{Main results}\label{sec.main}

In this section, we state the main results of this paper. To this end, we first introduce the necessary notation and terminology.

\subsection{Notation and terminology}

Let $\mathcal{X} \subset \mathbb{R}^{n+1}$ be a bounded $C^1$-smooth, $n$-dimensional manifold.
Assume $p\notin\mathcal X$ and $\sigma_{\mathcal X}(\mathcal X)<\infty$, where $\sigma_{\mathcal X}=\mathcal H^n|_{\mathcal X}$ is induced volume and $\mathcal H^n$ denotes Hausdorff measure normalized to agree with
Euclidean $n$-volume. In particular, $\mathcal H^2$ is surface area. All subsets and measures below are Borel. Whenever $d$ is a metric, we assume that $\mathcal{X}$ is separable and that its Borel sets agree with the inherited Borel sets; these properties hold for the concrete metrics used in the examples.

\begin{definition}
We say $\mathcal{X}$ is \textit{$k_p$-star-shaped} ($k \in \mathbb{N}$) with respect to $p$ 
if $p\in\mathbb R^{n+1}\setminus\mathcal X$ is fixed and for each ray $r_p$ emanating from $p$, the intersection of $r_p$ with $\mathcal{X}$ is nonempty and consists of at most $k$ points. 
\end{definition}

\begin{definition}
We say $\mathcal{X}$ belongs to the class of \emph{p-finite-star-shaped} manifolds, denoted as $\mathfrak{M}_p$, if $ \mathcal{X}$ is \textit{$k_p$-star-shaped} for some finite integer $k \geq 1$.   
\end{definition}








For $x \in \mathcal{X}$, let $r_p^x$ denote the ray starting from $p$ and passing through $x$. Define $i_p(x)$ as the \emph{layer index} of $x$, that is, the index of the intersection at $x$ between $r_p^x$ and $\mathcal{X}$, counting from $p$ outward. When $\mathcal{X}$ is $k_p$-star-shaped with respect to $p$, we have $i_p(x) \in \{1,\dots,k\}$. For each $i = 1,\dots,k$, the $i$-th layer is denoted by $L_p(i) \coloneqq \{x \in \mathcal{X} : i_p(x) = i\}$. Finally, define $\phi_p(x,y) \coloneqq 1_{\{i_p(x) = i_p(y)\}}$ as the indicator of whether $x$ and $y$ lie in the same layer.











Let $\mathcal{S}^n\subset \mathbb{R}^{n+1}$ be the unit sphere centred at the origin. For a fixed $t\in \mathcal{S}^n$, define the open hemisphere
\[H_t \coloneqq \{z\in \mathcal{S}^n: t\cdot z > 0\}.\]
Let $P_p\colon \mathcal{X}\to \mathcal{S}^n$ denote the radial projection with respect to the reference point $p$, namely,
\[
P_p(x)=\frac{x-p}{\|x-p\|}.
\]
Since $\mathcal{X}$ is $k_p$-star-shaped with respect to $p$, for each $i=1,\dots,k$, the restriction of $P_p$ to the Borel layer $L_p(i)$ is an injective Borel map with Borel inverse onto its image. It is a diffeomorphism on every regular radial branch. We denote 
\[
P_{p,i}\colon L_p(i) \rightarrow  P_p(L_p(i))\subseteq\mathcal{S}^n
\]
and the inverse mapping by
\[
P_{p,i}^{-1}\colon P_p(L_p(i))\subseteq\mathcal{S}^n\rightarrow L_p(i).
\]

Unless a relaxed radial-coverage condition is explicitly stated, we assume that $\mathcal{X} \in \mathfrak{M}_p$, and we write $k$, $P$ and $P_i$ for $k_p$, $P_p$ and $P_{p, i}$, respectively, when no ambiguity arises. For points $t \in \mathcal{S}^n$ that are not in the image of $P(L_p(i))$, we set $P_i^{-1}(t) = \emptyset$ by convention.

For each layer $L_p(i)$, let $\sigma_i$ denote the induced
$n$-dimensional Riemannian volume measure on $L_p(i)$.
Set $E_i=P(L_p(i))$, $m_i=(P_i)_\#\sigma_i$ and
$m=\sum_i m_i=P_\#\sigma_{\mathcal X}$ and $E=P(\mathcal X)$. Then
\[
\int_{L_p(i)} f(x)\,d\sigma_i(x)
=\int_{E_i}f(P_i^{-1}(t))\,dm_i(t).
\]
On a regular radial branch, $dm_i(t)=J_i(t)\,d\sigma_{\mathcal S^n}(t)$, where $J_i$ is the volume Jacobian of $P_i^{-1}$. The pushforward definition also applies at nonregular points and does not evaluate $f$ at an empty preimage.

The definition of $k_p$-star-shaped manifolds requires both complete pointwise inclusion of the manifold and complete radial coverage from the reference point $p$. We now relax these two requirements in two distinct ways. An \emph{almost $k_p$-star-shaped} manifold may differ from the underlying $k_p$-star-shaped manifold only on a null subset of the manifold. An \emph{essentially $k_p$-star-shaped} manifold remains $C^1$, but its radial projection may omit a null subset of directions on $\mathcal{S}^n$. The two cases represent, respectively, a measure-zero modification of an existing radial structure and an intrinsic failure of complete radial coverage.

\begin{definition}[almost $k_p$-star-shaped]
Let $\mathcal{X}_0$ be a $k_p$-star-shaped manifold with respect to
$p$. A subset $\mathcal{X}\subseteq\mathcal{X}_0$ is called an \emph{almost $k_p$-star-shaped} manifold with respect to $p$ if
\(\mathcal{X}_0\setminus\mathcal{X}\)
is a set of measure zero with respect to the induced $n$-dimensional Riemannian measure on $\mathcal{X}_0$.
\end{definition}

\begin{definition}[essentially $k_p$-star-shaped]\label{def.essential}
Let $\mathcal{X}\subset\mathbb{R}^{n+1}$ be a bounded $C^1$,
$n$-dimensional manifold satisfying the standing finite-volume assumptions. We say that $\mathcal{X}$ is an \emph{essentially $k_p$-star-shaped} manifold with respect to $p$ if the radial projection $P$ has full rank
$\sigma_{\mathcal X}$-almost everywhere on $\mathcal X$, and
there exists a null set $N\subseteq\mathcal{S}^n$
such that $P^{-1}(N)=\emptyset$ and $1\leq |P^{-1}(t)|\leq k$ for some fixed $k \in\mathbb N_+$ and all $t\in\mathcal{S}^n\setminus N$.
\end{definition}

\subsection{Main results}\label{subsec.main}

Let $\mathcal{X}$ denote a $k_p$-star-shaped $n$-dimensional manifold. For $x, y \in \mathcal{X}$, define
\begin{equation}\label{eq.metric}
\begin{split}
d(x,y)&\coloneqq J(x,y)+ D(x, y),
\end{split}
\end{equation}
where 
\begin{align}
D(x, y) &= \sum_{i=1}^k\int_{L_p(i)}
|1_{H_{P(z)}}(P(x))-1_{H_{P(z)}}(P(y))|^2\,{\rm d}\sigma_i(z) \nonumber\\
&=\int_{\mathcal S^n}|1_{H_t}(P(x))-1_{H_t}(P(y))|^2\,{\rm d}m(t) \label{eq.D}
\end{align}
is the angular part of $d$, and $J$ is the \emph{penalty function} satisfying that
\begin{enumerate}
    \item[(i)] $J$ is symmetric, $J(x, x) = 0$, $J(x,y)\geq 0$ and $J\,(x,z)+J\,(y,z)\geq J\,(x,y)$ for all $x,y,z\in \mathcal{X}$;
    \item[(ii)] $J(x, y) <\infty$ if $\Vert x-y \Vert <\infty$;
    \item[(iii)] $J$ satisfies the NT integral inequality, that is, for any finite signed measure $\nu$ with $\nu(\mathcal{X})=0$ and $\iint J\,d|\nu|\,d|\nu|<\infty$, 
    \begin{equation}\label{eq.J_negType}
    \int_{\mathcal{X}\times \mathcal{X}} J(x,y)\,{\rm d}\nu(x)\,{\rm d}\nu(y)\le 0.
    \end{equation}
    \item[(iv)] when the radial coverage is complete, $J$ separates points within the same fibre, that is,
\begin{equation}\label{eq.fibre-separation}
x\ne y,\quad P(x)=P(y)\quad\Longrightarrow\quad J(x,y)>0
\end{equation}
($J\equiv 0$ is allowed when $P$ is injective).
\end{enumerate}


Note that without the separation condition \eqref{eq.fibre-separation}, \eqref{eq.metric} defines only a pseudometric. For the almost $k_p$-star-shaped case, we take $J$ to be the restriction of the same penalty function defined on $\mathcal{X}_0$.

Simple choices of $J$ include $J\equiv 0$ when $P$ is injective, or $J(x,y)=c(1-\phi_p(x,y))$ with $c>0$, or $J(x,y)=|f(x)-f(y)|$ for any bounded measurable function $f$.





The first result of this paper is as follows.

\begin{theorem}\label{Thm.Xd}
    $(\mathcal{X},d)$ is a metric space of NT.
\end{theorem}

For a Borel subset $G$ of $\mathcal{X}$, Theorem~\ref{Thm.Gd1} states that $(G, d)$ is of SNT when $G$ satisfies the following \emph{half-space direction} (HSD) condition.

(HSD). There exists a hemisphere $H \subset \mathcal{S}^n$ and its preimage $P^{-1}(H) \subseteq \mathcal{X}$ such that
\[G=\{ x\in \mathcal{X}\;|\; \text{for each} \,\, t\in H, \text{choose one} \, \, x\in P^{-1}(t)\}.\]


\begin{theorem}\label{Thm.Gd1}
Assume that $G \subset \mathcal{X}$ is Borel and satifies (HSD). Then $(G,d)$ is of SNT.
\end{theorem}

In addition to (HSD), which represents a fully directional asymmetry, SNT also holds under the following conditions of injectivity (INJ) and relaxed opposite-directional asymmetry (ODA).

(INJ). $|G\cap P^{-1}(t)|\leq1$ for all $t\in \mathcal{S}^n$.

(ODA). There is at most one pair of antipodal points $\{x,x'\}$ of $\mathcal{S}^n$ such that $|G\cap P^{-1}(x)|\cdot|G\cap P^{-1}(x')| \neq 0$.


\begin{theorem}\label{Thm.Gd2}
For a Borel subset $G \subseteq \mathcal{X}$, we have
\begin{enumerate}
    \item[(i)] the metric subspace $(G,d)$ is of SNT if both (INJ) and (ODA) hold for $G$;
    \item[(ii)] when the penalty term $J \equiv 0$, the sufficient conditions (INJ) and (ODA) in (i) are also necessary.
\end{enumerate}

\end{theorem}

\begin{remark}\label{rmk.ThmGd2}
The conditions (INJ) and (ODA) are sufficient but not necessary in general. The necessity
depends on the penalty $J$, since the angular part $D$ in \eqref{eq.metric} cannot distinguish
two probability measures that differ only by a redistribution of mass within a single radial fibre. Indeed, assume
$P^{-1}(x)\cap G=\{a_1,\ldots,a_k\}$ with $k>1$ and fix $c>0$, and
choose \[
\mu_1(a_i) = \frac{c}{k} \quad \text{for all } i = 1,\dots,k,
\qquad
\mu_2(a_i) = \begin{cases}
0, & 1 \le i \le k-1, \\
c, & i = k.
\end{cases}
\]
These assignments yield identical angular data. Hence, when $J\equiv 0$, we have $E_d(\mu_1-\mu_2) = 0$ yet $\mu_1 \neq \mu_2$, and (INJ) becomes necessary. 
As for the necessity of (ODA), a concrete obstruction when $J\equiv 0$ is given by two distinct opposite-directional pairs: For two probability measures $\mu_1(P^{-1}(x)) = \mu_1(P^{-1}(-x)) = 1/2$ and $\mu_2(P^{-1}(y)) = \mu_2(P^{-1}(-y)) = 1/2$, with $x, y \in \mathcal{S}^n$ and $x \neq y$, we have 
$$(\mu_1 - \mu_2)(\pm x) = 1/2 \quad \text{and} \quad (\mu_1 - \mu_2)(\pm y) = -1/2,$$
and
\begin{align*}
2 E_d(\mu_1-\mu_2) \,=\, & d(P^{-1}(x), P^{-1}(-x)) + d(P^{-1}(y), P^{-1}(-y)) \\ &- \left[ d(P^{-1}(x), P^{-1}(y)) + d(P^{-1}(x), P^{-1}(-y)\right] \\ &-  \left[d(P^{-1}(-x), P^{-1}(y)) + d(P^{-1}(-x), P^{-1}(-y))\right] \\
=:\, & E_1 + E_2 - E_3 - E_4.
\end{align*}
When $J \equiv 0$, by \eqref{eq.D}, we can easily obtain $E_{\ell} = \int_{\mathcal S^n}{\rm d}m(t)$ for all $\ell = 1, 2, 3, 4$. Hence, $E_d(\mu_1-\mu_2) = 0$, yet $\mu_1 \neq \mu_2$.

\end{remark}




Theorem \ref{thm.typeIandII} extends the results in Theorems~\ref{Thm.Xd} and \ref{Thm.Gd2} to the almost and essentially $k_p$-star-shaped manifolds.


\begin{theorem}[almost or essentially $k_p$-star-shaped]\label{thm.typeIandII}
Let $\mathcal{X}_0$ be a $k_p$-star-shaped manifold, and $\mathcal{X}\subseteq\mathcal{X}_0$ be an almost $k_p$-star-shaped manifold with respect to $p$. Then
\begin{enumerate}
    \item[(i)]  $(\mathcal{X},d)$ is of NT;
    \item[(ii)]  for a Borel subset $G\subseteq\mathcal{X}$, we have
    \begin{enumerate}
        \item[(a)] if $G$ satisfies (INJ) and (ODA), the metric subspace $(G,d)$ is of SNT;
        \item[(b)] when $J \equiv 0$, the metric subspace $(G,d)$ is of SNT if and only if both (INJ) and (ODA) hold.
    \end{enumerate}
\end{enumerate}
The same conclusion holds for $\mathcal{X}$ that is essentially $k_p$-star-shaped.
\end{theorem}




Finally, we extend the above results to the manifolds with incomplete radial coverage on $\mathcal{S}^n$, a class that includes the torus in Example~\ref{ex.torus}. To this end, we need the following half-space signed measure (HSSM) condition, which requires that a probability measure on the projected directions is uniquely determined by its values on the half-space directions.

(HSSM). For every finite signed measure $\lambda$ concentrated on $P(G)$,
\begin{equation}\label{eq.restricted-injectivity}
\lambda(H_t)=0\quad\text{for }m\text{-almost every }t\in P(\mathcal X)
\quad\Longrightarrow\quad\lambda=0.
\end{equation}


\begin{theorem}[Incomplete radial coverage]
\label{thm:radial_completion}
Let $\mathcal X\subset\mathbb R^{n+1}$ satisfy $\sigma_{\mathcal X}(\mathcal X)<\infty$, and suppose
$P(\mathcal X)\ne\mathcal S^n$ and $|P^{-1}(t)|\leq k$ for every
$t\in\mathcal S^n$. Then $d$ in \eqref{eq.metric} is a negative-type pseudometric. It is a metric if
\begin{equation}\label{eq.general-separation}
x\ne y \quad\Longrightarrow\quad d(x, y)>0.
\end{equation}
\par\noindent\begin{minipage}{\textwidth}
Moreover, for a Borel subset $G\subseteq\mathcal{X}$ satisfying (INJ), (ODA) and (HSSM), 
$(G,d)$ is of SNT whenever $d$ is a metric.

\end{minipage}\par\end{theorem}

When $J=0$, the distinction between probability measures and signed measures becomes important. To state the necessary conditions for SNT, we relax (HSSM) to the following half-space probability measure (HSPM) condition

(HSPM). For any two probability measures $\mu_1',\mu_2'$ on $P(G)$,
\begin{equation}\label{eq:probability-hemispheres}
\mu_1'(H_t)=\mu_2'(H_t)\quad\text{for }m\text{-a.e. }t
\quad\Longrightarrow\quad \mu_1'=\mu_2'.
\end{equation}

The following proposition states precisely what is necessary.

\begin{proposition}\label{prop:zero-exact}
Let $\mathcal X$ satisfy the assumptions of Theorem~\ref{thm:radial_completion}, and let $J=0$
and $G$ be a nonempty Borel subset of $\mathcal X$.
Then $(G,d)$ is a metric space of SNT if and only if (INJ) and (HSPM) hold.
Moreover, SNT implies (ODA), but needs not imply (HSSM).
\end{proposition}

The difference in \eqref{eq:probability-hemispheres} always has total
mass zero. Equivalently, one may retain the formulation of (HSSM)
and require its implication only for $\lambda$ with
$\lambda(\mathcal S^n)=0$. Indeed, the positive and negative parts
of any nonzero such measure have the same mass; normalizing them
gives two distinct probabilities.

\begin{remark}\label{rmk:hsm-symmetry}
There is a useful case in which (HSSM) is necessary after all.
Suppose $m$ is nonzero, and $m$ and $R_\#m$ have the same null sets,
where $R(t)=-t$. If a finite signed measure $\lambda$ satisfies
$\lambda(H_t)=0$ $m$-a.e., then also $\lambda(H_{-t})=0$ $m$-a.e.
Moreover, \eqref{eq:boundary-null-addition} and Fubini give
$|\lambda|(\partial H_t)=0$ $m$-a.e. Hence
\[
\lambda(\mathcal S^n)=\lambda(H_t)+\lambda(H_{-t})=0
\qquad\text{for }m\text{-a.e. }t.
\]
Thus, under this symmetry assumption and with $J=0$, SNT is
equivalent to (INJ) and (HSSM); (ODA) follows from them.
This applies to the rotational toroidal hypersurface in
Example~\ref{ex.torus}, whose directional belt is invariant under
$t\mapsto-t$ and whose density $g$ in \eqref{eq.Q-density} is even.
\end{remark}


\subsection{Examples}\label{sec.example}

To illustrate the theoretical results in Section~\ref{sec.main}, we presents the examples of the ellipsoid and rotational torus that are manifolds with complete and incomplete radial coverage. Furthermore, we consider the product torus that has larger codimension than the rotational torus. Finally, as an illustration of Proposition~\ref{prop:zero-exact}, we construct an example where SNT holds while (HSSM) fails.

\begin{example}[Ellipsoidal radial manifold]

Consider the ellipsoid
\[
\mathcal E=
\left\{
x\in\mathbb R^{n+1}:\sum_{j=1}^{n+1}\frac{x_j^2}{a_j^2}=1,\quad a_j>0
\right\},
\]
with $p=0$. It admits the radial representation
\[
X(\omega)=r(\omega)\omega,\qquad \omega\in\mathcal S^n,
\]
where
\[
r(\omega)
=
\left(
\sum_{j=1}^{n+1}\frac{\omega_j^2}{a_j^2}
\right)^{-1/2}.
\]
Since each direction has a unique radial intersection, this is the
single-layer case where $L_p(1)=\mathcal E$ with $P_1(X(\omega))=\omega$.

For $x=X(\epsilon)$ and $y=X(\eta)$, the distance in \eqref{eq.metric}
takes the form
\[
d(x,y)
=
J(x,y)
+
\int_{\mathcal A(\epsilon,\eta)}
\rho(\omega)\,d\sigma_{\mathcal S^n}(\omega),
\]
where
\[
\mathcal A(\epsilon,\eta)
=
\{\omega\in\mathcal S^n:
(\omega\cdot\epsilon)(\omega\cdot\eta)<0\},
\]
and
\[
\rho(\omega)
=
r(\omega)^n
\sqrt{1+
|\nabla_{\mathcal S^n}\log r(\omega)|^2}.
\]
Here $\mathcal A(\epsilon,\eta)$ represents the set of hemispheres
separating the projected points, while $\rho$ is the density induced by
the radial deformation of $\mathcal E$.

More explicitly, differentiating $r$ on $\mathcal S^n$ gives
\[
\rho(\omega)=
\frac{\left(\sum_{j=1}^{n+1}\omega_j^2/a_j^4\right)^{1/2}}
{\left(\sum_{j=1}^{n+1}\omega_j^2/a_j^2\right)^{(n+2)/2}}.
\]
This positive continuous density is bounded above and below on the
sphere. Thus, with $J=0$, the distance is comparable with angular
distance and is a metric of NT for each $n \geq 1$.
Theorem~\ref{Thm.Gd2} gives SNT on Borel subsets
with at most one antipodal pair of directions, including any open
half-ellipsoid. The volume formula follows from the radial graph
metric $r^2I+\nabla r(\nabla r)^\top$.

To recover the explicit calculation in $\mathbb R^3$, set $n=2$,
$(a_1,a_2,a_3)=(a,b,c)$, and let
$\epsilon=(1,0,0)$,
$\eta=(\cos\alpha,\sin\alpha,0)$.
Then the separating condition reduces to
$\cos\theta\cos(\theta-\alpha)<0$,
and hence
\[
\begin{aligned}
d(x,y)=J(x,y) 
+ \int_{\pi/2}^{\pi/2+\alpha}\int_0^\pi \rho(\phi,\theta)\sin\phi\,d\phi d\theta 
+ \int_{3\pi/2}^{3\pi/2+\alpha}\int_0^\pi \rho(\phi,\theta)\sin\phi\,d\phi d\theta.
\end{aligned}
\]

In the isotropic case where $a=b=c=R$, the density becomes constant, i.e.,
$\rho(\omega)=R^2$, and therefore the slicing contribution depends only
on the angular separation:
\[
d(x,y)=J(x,y)+4R^2\alpha
\]
with the standard unnormalized sphere measure used here. For arbitrary
$n$, when all $a_j=R$, the corresponding term is
$R^n\sigma_{\mathcal S^n}(\mathcal S^n)\alpha/\pi$. Thus the ellipsoidal case can be
viewed as a weighted spherical separation, where the radial geometry
modifies the pushforward measure on $\mathcal{S}^n$ but preserves the underlying spherical
partition.
\end{example}

\begin{example}[Rotational toroidal hypersurface]\label{ex.torus}
Consider the rotational toroidal hypersurface
\[
\mathcal Q^n_{R,r}=\{(x',x_{n+1})\in\mathbb R^n\times\mathbb R: (\|x'\|-R)^2+x_{n+1}^2=r^2\},\qquad n\geq2,\quad R>r>0,
\]
and take $p=0$, which lies in the central hole. Unlike the ellipsoid case, the radial projection is neither globally one-to-one nor onto $\mathcal S^n$, so $\mathcal Q^n_{R,r}$ provides an example in which the layer structure and incomplete radial
coverage occur simultaneously. 
For $n=2$, this is the standard torus in $\mathbb R^3$. For $n>2$, this hypersurface is diffeomorphic to $\mathcal S^{n-1}\times\mathcal S^1$, not to the product of $n$ circles. The calculation below includes this case without a separate proof.

\emph{Layer structure.} Its parameterisation is
\begin{equation}\label{eq.torusPara}
    ((R+r\cos v)u,r\sin v),\qquad (u,v)\in\mathcal S^{n-1}\times\mathcal S^1.
\end{equation}
For a unit direction $t=(\sqrt{1-z^2}\,u,z)$ with $u\in\mathcal S^{n-1}$ and $z=t_{n+1}$, set $a=r/R$, $q=\sqrt{1-z^2}$ and $s(z)=\sqrt{r^2-R^2z^2}$. The radial equation is $$\rho^2-2Rq\rho+R^2-r^2=0,$$
with roots $\rho_\pm=Rq\pm s(z)$. Both roots are positive when $|z|<a$. At $|z|=a$ they coincide; the tangent set is a smooth $(n-1)$-dimensional set and has zero
$\mathcal H^n$ measure, assigned to the first layer only. Hence, over the open belt 
$$B_{R,r}^\circ=\{t\in\mathcal S^n: |t_{n+1}|<a\},$$
$\mathcal Q^n_{R,r}$ admits the two radial branches $L_p(\pm)$. The parametrisation \eqref{eq.torusPara} gives
$$\dd\mathcal H^n=r(R+r\cos v)^{n-1} \dd\sigma_{\mathcal S^{n-1}}(u)\dd v,$$
so the total volume is finite and equals
\begin{equation}\label{eq.Q-volume}
    r\sigma_{\mathcal S^{n-1}}(\mathcal S^{n-1})\int_0^{2\pi}(R+r\cos v)^{n-1}\dd v.
\end{equation}

\emph{The metric.} The unit normal is $N=(u\cos v,\sin v)$. On a radial branch $L_p(+)$ or $L_p(-)$, $|N\cdot t|=s(z)/r$, and the radial projection has Jacobian $|N\cdot t|/\rho^n$; equivalently, this follows by taking the determinant of the radial graph metric. Hence the two directional densities and their sum are
\begin{equation}\label{eq.Q-density}
g_\pm(t)=\frac{r\rho_\pm(t)^n}{s(t_{n+1})},\qquad g(t)=\frac{r[\rho_-(t)^n+\rho_+(t)^n]}{s(t_{n+1})}.
\end{equation}
They are positive on the open belt $B_{R,r}^\circ$ and integrable over its closure $B_{R,r}$. Indeed,
\[
\dd\sigma_{\mathcal S^n}(t)
=(1-z^2)^{(n-2)/2}\dd\sigma_{\mathcal S^{n-1}}(u)\dd z,
\]
and the only endpoint singularity in \eqref{eq.Q-density} is of order $(a-|z|)^{-1/2}$. Its integral is precisely the volume in \eqref{eq.Q-volume}.

For $c>0$, let $J(x,y)=c1_{\{i_p(x)\ne i_p(y)\}}$ that satisfies the separation condition~\eqref{eq.fibre-separation}.
Then the $n$-dimensional metric is
\begin{equation}\label{eq.Q-metric}
d_Q(x,y)=J(x,y)+\int_{-a}^{a}\int_{\mathcal S^{n-1}}
1_{\{(t\cdot P(x))(t\cdot P(y))<0\}} w(t)(1-z^2)^{(n-2)/2}\dd\sigma_{\mathcal S^{n-1}}(u)\dd z.
\end{equation}


\emph{Negative type.} Here the separation condition \eqref{eq.general-separation} can be proved directly. Let $t_1=(h_1,z_1)$ and $t_2=(h_2,z_2)$ be two distinct directions in $B_{R,r}$, with $h_1,h_2\in\mathbb R^n\setminus\{0\}$. If $h_1$ and $h_2$ are not positive multiples, choose a horizontal unit vector $\omega\in\mathcal S^{n-1}$ such that $\omega\cdot h_1>0>\omega\cdot h_2$. Then the normal $(\omega,0)$ strictly separates $t_1$ and $t_2$. Otherwise $h_i=\alpha_i u$, $i = 1, 2$, for some $\alpha_i>0$ and $\|u\|=1$. Since $t_1$ and $t_2$ are two distinct directions, we have $z_1/\alpha_1\ne z_2/\alpha_2$. Choose $\ell$ strictly between these ratios and a unit vector $v\perp u$. For sufficiently small $\varepsilon\ne0$, the normal proportional to $(v-\varepsilon \ell u,\varepsilon)$ has opposite inner products with $t_1$ and $t_2$, and lies in $B_{R,r}^\circ$. In either case, the set of separating normals is open and intersects $B_{R,r}^\circ$, so the second term in \eqref{eq.Q-metric} is positive. The penalty term $J$ separates points in the same fibre but different layers. Thus $d_Q$ is a metric, and $(\mathcal Q^n_{R,r}, d_Q)$ is of NT by Theorem~\ref{thm:radial_completion}.

Moreover, Proposition~\ref{prop:torus-positive-patch} states that a positive-volume patch of the outer or inner layer of $\mathcal Q^n_{R,r}$ is of SNT.


\end{example}

\begin{proposition}[SNT of a positive-volume patch]\label{prop:torus-positive-patch}
For every $n\geq2$, choose $0<\delta<a$ and $0<\eta<1$, and define
\[
U=\{t\in\mathcal S^n:|t_{n+1}|<a-\delta,\ t_1>\eta\},
\qquad G_+=P_+^{-1}(U), \qquad G_-=P_-^{-1}(U).
\]
Then $\mathcal H^n(G_+)>0$ and $\mathcal H^n(G_-)>0$, and the restriction of $d_Q$ to $G_+$ or $G_-$
is SNT.
\end{proposition}

\begin{example}[Product torus and higher codimension]
Consider the product torus $\mathcal{T}^n_a =\prod_{j=1}^n\mathcal S^1(a_j)\subset\mathbb R^{2n}$ with $a_j >0$ and $n\geq 1$.  Different from $\mathcal Q^n_{R,r}$, $\mathcal{T}^n_a$ has codimension $n$, rather than $1$. For $\theta \in \mathcal{T}^n_a$, its parameterisation is 
$$X(\theta)=(a_1\cos\theta_1,a_1\sin\theta_1,\ldots, a_n\cos\theta_n,a_n\sin\theta_n),$$
and the induced volume and radial map are
$$\dd\mathcal H^n=\left(\prod_{j=1}^n a_j\right) \dd\theta_1\cdots\dd\theta_n,\qquad P(X(\theta))=X(\theta)/A,$$
where $A = (\sum_j a_j^2)^{1/2}$ is the standardisation constant. For $\theta, \vartheta\in \mathcal{T}^n_a$ and $\varepsilon>0$, let 
\begin{equation}\label{eq.Jtorus}
J(\theta, \vartheta) = \varepsilon\|X(\theta)-X(\vartheta)\| = 2\varepsilon\left(\sum_{j=1}^n a_j^2 \sin^2\frac{\theta_j-\vartheta_j}{2}\right)^{1/2},
\end{equation}
$$K(\phi,\theta)=X(\phi)\cdot X(\theta) = \sum_{j=1}^n a_j^2\cos(\phi_j-\theta_j)$$
and
$$D_T(\theta,\vartheta)= \left(\prod_{j=1}^n a_j\right)\int_{[0,2\pi)^n} \left|1_{\{K(\phi,\theta)>0\}}- 1_{\{K(\phi,\vartheta)>0\}}\right|^2\dd\phi,$$ 
which is a finite negative-type metric. Then $d$ in \eqref{eq.metric} takes the form
\begin{equation}\label{eq.product-strong}
d_{T,\varepsilon}(\theta,\vartheta) =D_T(\theta,\vartheta)+ 2\varepsilon\left(\sum_{j=1}^n a_j^2 \sin^2\frac{\theta_j-\vartheta_j}{2}\right)^{1/2},
\end{equation}
 
\end{example}

Proposition~\ref{prop:product-torus} establishes a strong-negative-type result for every genuine $n$-torus $\mathcal{T}^n_a$ without identifying it with $\mathcal S^{n-1}\times\mathcal S^1$ or applying the hypersurface coverage theorem in the wrong ambient dimension. This example is not intended to illustrate the condition (ODA), as $\mathcal{T}^n_a$ contains uncountable many opposite directions; rather, it shows a different route to SNT: one may add a penalty term that is itself of SNT.


\begin{proposition}\label{prop:product-torus}
For every $\varepsilon>0$, the metric $d_{T,\varepsilon}$ in \eqref{eq.product-strong} is SNT on the whole product torus $\mathcal{T}^n_a$ and induces its ordinary topology.
\end{proposition}


\begin{example}\label{ex:hsm-not-necessary}
As an illustration of Proposition~\ref{prop:zero-exact}, we construct a set of six points for which the only ambiguity in the hemisphere data changes the total mass. Such an ambiguity violates
(HSSM), but cannot occur between two probabilities.

Let
\[
M=\begin{pmatrix}
1&1&0&0&0&0\\
1&1&1&0&1&0\\
0&1&1&0&0&0\\
0&0&0&1&1&0\\
0&1&0&1&1&1\\
0&0&0&0&1&1
\end{pmatrix},\qquad
c=(1,-1,1,1,-1,1)^{\mathsf T}.
\]
The kernel of $M$ is spanned by $c$, whose entries sum to $2$.
Choose unit vectors $z_1,\ldots,z_6\in\mathbb R^6$ such that, for $i\ne j$, $z_i\cdot z_j=1/10$ when $M_{ij}=1$ and
$z_i\cdot z_j=-1/10$ otherwise. These vectors exist and are linearly independent: the prescribed Gram matrix \(B\) has diagonal entries $1$ and off-diagonal entries $\pm1/10$, so $B$ is positive definite by strict diagonal dominance. Hence, we can find a matrix \(Z\) such that \(B=Z^{\mathsf T}Z\). 


Take $p=0$ and set
\[
\mathcal X=\left\{t\in\mathcal S^5:
t\cdot z_j\ne0\ (j=1,\ldots,6)\;,\;
\sum_{j=1}^6c_j1_{H_t}(z_j)=0\right\},
\qquad G=\{z_1,\ldots,z_6\}.
\]
Then $\mathcal X \subset \mathcal S^5$ is a bounded smooth manifold with finite volume and incomplete radial coverage. Its radial projection is the identity, so $m$ is simply spherical volume restricted to $\mathcal X$. Since $Mc=0$, each $z_i$ belongs to $\mathcal X$; in particular, the volume is positive. By construction, the nonzero measure $\sum_jc_j\delta_{z_j}$ vanishes on $H_t$ for every $t\in\mathcal X$. Thus (HSSM) fails.

On the other hand, let $\nu=\sum_ja_j\delta_{z_j}$ have total mass zero and suppose $E_d(\nu)=0$.
Its hemisphere values vanish $m$-a.e. In a small neighbourhood of $z_i$, the signs of all inner products are fixed, and this value is just $(Ma)_i$. Each neighbourhood has positive volume, so we obtain $Ma=0$. Consequently $a$ is a multiple of $c$. Since the entries of $a$ sum to zero and those of $c$ do not, $a=0$.
This proves both point separation and SNT on $G$.

Next, we verify that $d$ is also a metric on the whole manifold $\mathcal{X}$. The same formula for $d$ in \eqref{eq.D} separates all points
of $\mathcal S^5$.
It suffices to find, for each $u, v \in \mathcal{X}$ with $u\ne v$, a normal $t\in\mathcal X$
with $(u\cdot t)(v\cdot t)<0$: by continuity, this inequality holds on a set of positive
$m$-measure, so $d(u,v)>0$.

Since $z_1,\ldots,z_6$ form a basis, any nonzero numbers
$r_1,\ldots,r_6$ determine a unique $w$ with
$w\cdot z_\ell=r_\ell$. Let \(t={w}/{\|w\|}\), thus
\[
t\cdot z_\ell=\frac{r_\ell}{\|w\|}\;\;,\;\;
t\in\mathcal X\ \Longleftrightarrow\
\sum_{j:\,r_j>0}c_j=0\;\; (r_i\ne 0 \text{ for all } i)
\]
Taking all $r_\ell<0$ obviously satisfies this condition. Instead, we can also arrange $r_i>0>r_h$ for any $i\ne h$:
choose $k\ne h$ with $c_k=-c_i$, and make only $r_i,r_k$
positive, so their contributions cancel. The magnitudes $|r_\ell|$ are unrestricted.

First, we suppose that $u$ is not parallel to any $z_j$, and write
$u=\sum_\ell b_\ell z_\ell$, with at least two nonzero coefficients.
Now we seek $t\in\mathcal X$ with
$u\cdot t=0$, or equivalently $\sum_\ell b_\ell r_\ell=0$.
If the $b_\ell$ have both signs, take all $r_\ell<0$.
Otherwise, choose two nonzero coefficients $b_i,b_h$ and arrange
$r_i>0>r_h$ as above. In either case, the sum has both positive
and negative terms. Scaling the $r_\ell$ corresponding to the
positive terms by a common positive factor makes their sum
cancel the negative terms, without changing any signs.
Hence $\mathcal X\cap u^\perp$ is nonempty.

If $v$ is not parallel to $u$, choose
$t_0\in\mathcal X\cap u^\perp$ with $v\cdot t_0\ne0$.
This is possible because $\mathcal X\cap u^\perp$ is open in the
great sphere $\mathcal S^5\cap u^\perp$, whereas imposing
$v\cdot t=0$ gives a lower-dimensional great sphere. Set
\[
t_\varepsilon=\frac{t_0+\varepsilon u}{\sqrt{1+\varepsilon^2}},
\qquad
u\cdot t_\varepsilon=\frac{\varepsilon}{\sqrt{1+\varepsilon^2}}.
\]
For small $|\varepsilon|$, we still have $t_\varepsilon\in\mathcal X$
and $v\cdot t_\varepsilon$ has the sign of $v\cdot t_0$.
Then we can choose a $\varepsilon$ that $\varepsilon(v\cdot t_0)<0$, therefore gives
$(u\cdot t_\varepsilon)(v\cdot t_\varepsilon)<0$.
If $v=-u$, any $t\in\mathcal X\setminus u^\perp$ separates them;
such a $t$ exists since $\mathcal X$ is nonempty and open.

By symmetry, only distinct pairs in
$\{\pm z_1,\ldots,\pm z_6\}$ remain.
For $z_i,-z_h$, taking all $r_\ell<0$ gives
$z_i\cdot t<0<(-z_h)\cdot t$.
For $z_i,z_h$ with $i\ne h$, choosing $r_i>0>r_h$ gives
$z_i\cdot t>0>z_h\cdot t$; the same $t$ separates $-z_i,-z_h$.
Thus $d(u,v)>0$ for all distinct $u,v\in\mathcal S^5$.

For each $u\in\mathcal S^5$, the section
$\mathcal X\cap u^\perp$ is $m$-null, so dominated
convergence makes $d$ continuous in the spherical topology.
The identity from the compact sphere to this metric sphere is
therefore a homeomorphism. Restricting to $\mathcal X$ verifies
the standing topology and Borel assumptions.
\end{example}

\section{Proofs}\label{sec.proof}

This section contains proofs of the results in Section~\ref{sec.main}. The following lemma is essential for proving that $d$ is a metric.

\begin{lemma}\label{lem.volume-positive}
Let $P_i^{-1}:P(L_p(i))\rightarrow L_p(i)$ be the inverse radial
projection of the $i$-th layer. For every measurable set
$A\subseteq P(L_p(i))$ with
$\sigma_{\mathcal S^n}(A)>0$,
we have
$\sigma_i(P_i^{-1}(A))>0$.
\end{lemma}
\begin{proof}
Let $A\subseteq P(L_p(i))$ be Borel with positive spherical volume.
Since $P$ is locally Lipschitz away from $p$, the area formula gives
\begin{equation}\label{eq.area-formula}
\sigma_{\mathcal S^n}(A)=\int_{P_i^{-1}(A)}\operatorname{Jac}_nP(x)\,d\sigma_i(x).
\end{equation}
Consequently $\sigma_i(P_i^{-1}(A))>0$, including at nonregular layers.
For completeness, the original radial determinant calculation on a
regular branch is as follows. In coordinates centred at $p$, write
$x=r(\theta)\theta$ with $r>0$ smooth.

Taking local coordinates $(u^1,\dots,u^n)$ on $L_p(i)$, we have
\[x(u^1,\dots,u^n) = r(\theta(u))\,\theta(u),\]
and the tangent vectors are
\[\frac{\partial x}{\partial u^{\alpha}} = \frac{\partial r}{\partial u^{\alpha}}\,\theta + \frac{\partial\theta}{\partial u^{\alpha}} r.\]
Hence the induced metric (the first fundamental form) has components
\[g_{\alpha\beta}^x = \frac{\partial r}{\partial u^{\alpha}}\frac{\partial r}{\partial u^{\beta}} + r^2 h_{\alpha\beta}^x,\]
with $h_{\alpha\beta}^x \coloneqq \frac{\partial\theta}{\partial u^{\alpha}}\cdot\frac{\partial\theta}{\partial u^{\beta}}$, where the cross terms vanish due to $\theta\cdot\frac{\partial\theta}{\partial u^{\alpha}}=0$. 
In matrix form,
\[
g^x = r^2 H + (\nabla r)(\nabla r)^\top,
\]
where $H$ is the metric tensor of the unit sphere $\mathcal{S}^n$ (in the coordinates $u^{\alpha}$). Since $r(\theta) > 0$ and $H$ is positive definite, $g^x$ is invertible. Applying elementary operations (or applying directly the matrix determinant lemma) to the determinant $\det(g^x)$ yields
\begin{align*}
\det(g^x) &=
\begin{vmatrix}
r^2H+(\nabla r)(\nabla r)^\top & \vec{0} \\
\vec{0} &  1
\end{vmatrix}
=
\begin{vmatrix}
    r^2H+(\nabla r)(\nabla r)^\top & \vec{0} \\
(\nabla r)^\top &  1
\end{vmatrix}
=
\begin{vmatrix}
     r^2H & -(\nabla r) \\
(\nabla r)^\top &  1
\end{vmatrix}
\\
&=\begin{vmatrix}
       r^2H & \vec{0} \\ 
(\nabla r)^\top &  1+(\nabla r)^\top(r^2H)^{-1}(\nabla r)
\end{vmatrix}
=\det(r^2 H)\bigl(1 + (\nabla r)^\top (r^2 H)^{-1} (\nabla r)\bigr)
\end{align*}
Using the standard angular coordinates $x = (\theta^1,\dots,\theta^n)$ on $\mathcal{S}^n$, we obtain
\[H = \operatorname{diag}\Bigl(1,\; \sin^2(\theta^1),\; \sin^2(\theta^1)\sin^2(\theta^2),\; \dots,\; \prod_{i=1}^{n-1}\sin^2(\theta^i)\Bigr).\]
Consequently,
\[{\rm d}\sigma_{}\bigl(P_i^{-1}({t})\bigr)
= r^n \sqrt{\det(H)}\;\sqrt{1 + \frac{(\nabla r)^\top H^{-1}(\nabla r)}{r^2}}\; {\rm d}t.\]
Therefore,
$\sigma_i(P_i^{-1}(A))>0$ whenever $\sigma_{\mathcal S^n}(A)>0$, as asserted.

\end{proof}


\begin{proof}[\normalfont\bfseries Proof of Theorem~\ref{Thm.Xd}] We first verify that $d$ is a metric on $\mathcal{X}$.
Let $D=d-J$ and $E_d(\nu)=\iint d(x,y)\,{\rm d}\nu(x){\rm d}\nu(y)$,
and define $E_J$ analogously. Note that $d$ is finite due to $0\leq D\leq m(\mathcal S^n)$, and it is symmetric and
zero on the diagonal. Now we verify positivity of $d$ for $x\ne y$. If $P(x)\neq P(y)$, the set 
$$\{t\in \mathcal{S}^n: \textcolor{black}{|1_{H_t}(P(x))-1_{H_t}(P(y))|^2=1}\}$$
is a nonempty open subset\footnote{Indeed, it is a union of two spherical lunes.} of \(\mathcal S^n\), hence has positive spherical volume. Since $P(L_p(1))=\mathcal S^n$,
Lemma~\ref{lem.volume-positive} gives $D(x,y)>0$. If $P(x)=P(y)$
but $x\neq y$, \eqref{eq.fibre-separation} gives $d(x,y)>0$. Next, we verify the triangle inequality. For $a,b,c\in\{0,1\}$,
\[
|a-b|^2=|a-b|\leq|a-c|+|c-b|=|a-c|^2+|c-b|^2.
\]
Applying this pointwise and integrate, and adding the triangle inequality
for $J$ gives the triangle inequality for $d$.


Let $\nu$ be a finite signed measure with $\nu(\mathcal X)=0$ and integrable $J$-energy $E_J(\nu)$. Finiteness
of $m$ justifies Fubini's theorem for the angular term $D$. With
$h_t(x)= 1_{H_t}(P(x))$, expansion gives
\begin{align}
\iint|h_t(x)-h_t(y)|^2\dd\nu(x)\dd\nu(y)
&=\iint(h_t(x)+h_t(y)-2h_t(x)h_t(y))\dd\nu(x)\dd\nu(y) \nonumber\\
&=-2\left(\int h_t\dd\nu\right)^2. \label{eq.ht}
\end{align}
To spell out the cancellation, put $a_t=\int h_t\,d\nu$. Since
$h_t^2=h_t$, the first two terms in this expansion are
\[
\iint h_t(x)\,d\nu(x)d\nu(y)=a_t\nu(\mathcal X)=0,
\qquad
\iint h_t(y)\,d\nu(x)d\nu(y)=\nu(\mathcal X)a_t=0,
\]
whereas the mixed term is $-2a_t^2$. Also, the layer partition gives
\[
a_t=\nu(P^{-1}(H_t))=\sum_{j=1}^k\nu(P_j^{-1}(H_t)).
\]
Thus the layer decomposition is performed before taking the square, so
cross terms between different layers are retained.
The two linear terms vanish because $\nu(\mathcal X)=0$. Therefore,
\begin{align}
E_d(\nu)&=E_J(\nu)-2\int_{\mathcal S^n}\bigl[\nu(P^{-1}(H_t))\bigr]^2\dd m(t)
\label{eq.I}\\
&=E_J(\nu)-2\sum_{i=1}^k\int_{\mathcal S^n}
\left[\sum_{j=1}^k\nu(P_j^{-1}(H_t))\right]^2\dd m_i(t)\leq0.
\nonumber
\end{align}
This proves NT for finitely supported $\nu$. 

Indeed, for
$\nu=\mu_1-\mu_2$ with $\mu_\ell\in\mathcal P_1(\mathcal X,d)$, the
energy is integrable by the triangle inequality and $J\leq d$.
More precisely, fixing $x_0\in\mathcal X$ and writing
$\eta=\mu_1+\mu_2$, one has $|\nu|\leq\eta$ and
\[
\iint d(x,y)\,{\rm d}|\nu|(x) {\rm d}|\nu|(y)
\leq 2\eta(\mathcal X)\int d(x,x_0)\,{\rm d}\eta(x)<\infty.
\]
This also bounds the absolute $J$-energy, since $0\leq J\leq d$.
Both terms on the right of \eqref{eq.I} are nonpositive. Their sum
can therefore vanish only when $E_J(\nu)=0$ and the integral of
$a_t^2$ vanishes.
Equality in \eqref{eq.I} implies
\begin{equation}\label{eq.equality-data}
\nu(P^{-1}(H_t))=0\qquad\text{for }m\text{-almost every }t.
\end{equation}
\end{proof}

Before proving Theorem~\ref{Thm.Gd1}, we need the following result.

\begin{lemma}\label{lem.injective}
Let $H \subset \mathcal{S}^n$ be an open hemisphere. For a finite signed measure $\mu$ on $P^{-1}(H)$, 
 $t \in \mathcal{S}^n$, and $A\subseteq H$, define measures 
\begin{align*}
\mu_{(i)}'(A) & \coloneqq \mu(P^{-1}_i(A)),\;b_{\mu_{(i)}'}(t) \coloneqq \mu_{(i)}'\,(H\cap H_t), \quad i = 1, \ldots, k \\
\mu'(A) & \coloneqq \mu(P^{-1}(A)),\;b_\mu'(t) \coloneqq \mu '\,(H\cap H_t).
\end{align*}
Then {\rm (i)} both the maps $\mu_{(i)}' \mapsto b_{\mu_{(i)}'}$ and $\mu' \mapsto b_{\mu'}$ are injective, and {\rm (ii)} if $D$ is a dense subset of $\mathcal{S}^n$, then $\mu_{(i)}' \mapsto b_{\mu_{(i)}'}|_D$ and $\mu' \mapsto b_{\mu'}|_D$ are also injective.
\end{lemma}

\begin{proof}
Part (i). Since $\mu_i'$ is a finite signed measure on $H$, for every $A\subseteq H$, 
\[|\mu'(A)|= \left\vert \sum_{i=1}^k \mu(P^{-1}_i (A))\right\vert \leq \sum_{i=1}^k |\mu(P^{-1}_i (A))|<\infty;\]
that is, $\mu'$ is also a finite signed measure on $H$. Then, by Theorem 1.3 in \cite{lyons2020}, both the maps $\mu_{(i)}' \mapsto b_{\mu_{(i)}'}$ and $\mu' \mapsto b_{\mu'}$ are injective.



Part (ii). Let $\lambda$ be the difference of two finite
signed measures on $H$ whose hemisphere data agree on $D$.
Write $H=\{x \in \mathcal{S}^n: a\cdot x>0\}$ for some $a \in \mathcal{S}^n$. 

First, if $t\in \mathcal{S}^n$ satisfies $|\lambda|(H\cap\partial H_t)=0$, choose a sequence $t_{\ell}\in D$ that converges to $t$. By dominated convergence, $\lambda(H\cap H_{t_\ell})\to \lambda(H\cap H_t)$. Since $\lambda(H\cap H_{t_\ell})=0$ for all $\ell$, we obtain $\lambda(H\cap H_t)=0$.

Now, let $t\in \mathcal{S}^n$ be arbitrary. Choose $\varepsilon_j\downarrow0$ so that
the level sets $$\{x\in H: (t\cdot x)/(a\cdot x)=\varepsilon_j\}$$
have zero $|\lambda|$-measure; this is possible because only countably many levels can have positive $|\lambda|$-measure. The normals proportional to $t-\varepsilon_j a$
are continuity normals, and their hemispheres on $H$ increase to
$H\cap H_t$. Continuity from below gives $\lambda(H\cap H_t)=0$.

Thus $\lambda(H\cap H_t)=0$ for every $t\in \mathcal{S}^n$. Part (i) now implies $\lambda=0$. The argument applies to each layer measure and also to their sum.
\end{proof}

   



\begin{proof}[\normalfont\bfseries Proof of Theorem~\ref{Thm.Gd1}]
Let $\nu=\mu_1-\mu_2$ for two probability measures $\mu_1,\mu_2 \in\mathcal P_1(G,d)$, and suppose
\begin{equation}\label{eq.STNeq}
E_d(\nu)=\iint_{G\times G}d(x,y)\dd\nu(x)\dd\nu(y)=0.
\end{equation}
Extend $\nu$ by zero to $\mathcal X$. Equations \eqref{eq.I} and
\eqref{eq.equality-data}, together with $\sigma_{\mathcal S^n}\ll m$, imply
$\nu(P^{-1}(H_t))=0$ for spherical volume-almost every $t$.
Indeed, every nonempty open subset of $\mathcal S^n$ has positive
spherical volume and therefore intersects this full-volume set. Moreover,
the pushforward probabilities $\mu'_\ell=P_\#\mu_\ell$, $\ell=1,2$, satisfy
\[
\mu'_1(H\cap H_t)-\mu'_2(H\cap H_t)
=\nu(P^{-1}(H_t))=0
\]
for every $t$ in that dense set, because both are concentrated on $H$.
The measure $\lambda=P_\#\nu$ is
concentrated on $H$, and $b_\lambda(t)=0$ there. By
Lemma~\ref{lem.injective}, $P_\#\nu=0$. The injective Borel map
$P|_G$ has a Borel inverse onto its image, so $\nu=0$.
For clarity, the last implication can be checked on Borel sets.
If $A\subseteq G$ is Borel, the Borel inverse of $P|_G$ ensures that
$P(A)$ is Borel, and injectivity gives
\[
\mu_\ell(A)=\mu'_\ell(P(A)),\qquad \ell=1,2.
\]
Hence equality of the pushforward measures yields equality on every
Borel subset of $G$, which is precisely $\mu_1=\mu_2$. Conversely,
$\nu=0$ gives \eqref{eq.STNeq}. This proves SNT.
\end{proof}

Let $R \colon \mathcal{S}^n \to \mathcal{S}^n, x \mapsto -x$ denote the antipodal reflection. For a signed measure $\gamma$ on $\mathcal{S}^n$, define its \emph{antisymmetrised} measure $\overline{\gamma}$ by $\overline{\gamma} \coloneqq \gamma - R_\#\gamma$, where $R_\#\gamma$ is the pushforward of $\gamma$ by $R$, i.e., $(R_\#\gamma)(A) = \gamma(R^{-1}(A))$ for measurable $A \subseteq \mathcal{S}^n$. A \emph{partitioning hemisphere} $K$ is a Borel subset of $\mathcal{S}^n$ such that $K$ and $R(K)$ partition $\mathcal{S}^n$ and the interior of $K$ is a hemisphere.

We need the following result to prove Theorem~\ref{Thm.Gd2}.

\begin{lemma}\label{lem:antisymm}
Let $\mu_1, \mu_2$ be two probability measures on $\mathcal{X}$, and $\mu_\ell' = \mu_\ell \circ P^{-1}$, $\ell = 1, 2$. Then the following hold:
\begin{enumerate}
\item[(i)] If $\mu_1'(K) = \mu_2'(K)$ for every partitioning hemisphere $K \subset \mathcal{S}^n$, then the antisymmetrised measures satisfy $\overline{\mu_1'} = \overline{\mu_2'}$.
\item[(ii)]  If $\mu_1'(H_t)=\mu_2'(H_t)$ for every $t$ in a dense subset $D$ of $\mathcal{S}^n$, then $\overline{\mu_1'}=\overline{\mu_2'}$.
\end{enumerate}
\end{lemma}

\begin{proof}
Proof of part (i) consists of three steps.

\textit{Step 1. Construct an $(n-1)$-dimensional great sphere $\mathcal{A}^{n-1} \subset \mathcal{S}^n$ such that $\mu_1'(\mathcal{A}^{n-1}) = \mu_2'(\mathcal{A}^{n-1}) = 0$.} 

For each fixed $x\in\mathcal S^n$, the set of normals $u$ with
$u\cdot x=0$ has spherical measure zero. By Fubini's theorem,
\[
\int_{\mathcal S^n}(\mu'_1+\mu'_2)(\partial H_u)\,d\sigma_{\mathcal S^n}(u)=0;
\]
hence, $(\mu'_1+\mu'_2)(\partial H_u) = 0$ for $\sigma_{\mathcal S^n}$-almost every $u$. Choose such a normal $u$ and put $\mathcal A^{n-1}=\partial H_u$.
Since $\mu_\ell'\le\mu_1'+\mu_2'$, we have $\mu_\ell'(\mathcal A^{n-1})=0$ for $\ell=1,2$.

\textit{Step 2. $\overline{\mu_1'}(H\cap K) =\overline{\mu_2'}(H\cap K)$ for any partitioning hemisphere $K$ of $\mathcal{S}^n$.}

Let $H \subset \mathcal{S}^n$ be one of the two open hemispheres determined by $\mathcal{A}^{n-1}$ (so that $\mathcal{S}^n = H \cup \mathcal{A}^{n-1} \cup R(H)$). Choose a Borel set $\tilde{\mathcal{A}}^{n-1} \subset \mathcal{A}^{n-1}$ such that
\[\tilde{\mathcal{A}}^{n-1} \cup R(\tilde{\mathcal{A}}^{n-1}) = \mathcal{A}^{n-1} \quad \text{and} \quad \tilde{\mathcal{A}}^{n-1} \cap R(\tilde{\mathcal{A}}^{n-1}) = \emptyset;\]
for instance, choose $\tilde{\mathcal{A}}^{n-1}$ that contains exactly one point from each antipodal pair of ${\mathcal{A}}^{n-1}$. Then $H \cup \tilde{\mathcal{A}}^{n-1}$ is a partitioning hemisphere and $\mu_1'(\tilde{\mathcal{A}}^{n-1})=\mu_2'(\tilde{\mathcal{A}}^{n-1})=0$. Thus we have
$\mu_1'(H) = \mu_2'(H)$.

Now for any partitioning hemisphere $K$ of $\mathcal{S}^n$, note that
\begin{align}
\overline{\mu_\ell'}(H\cap K) &= {\mu_\ell'}(H\cap K) - {\mu_\ell'}(R(H)\cap R(K)) \nonumber\\
&={\mu_\ell'}(H\cap K) + {\mu_\ell'}(H\cap R(K)) -{\mu_\ell'}(H\cap R(K)) - {\mu_\ell'}(R(H)\cap R(K)) \nonumber\\
&={\mu_\ell'}(H) - {\mu_\ell'}(R(K)) = {\mu_\ell'}(H) + {\mu_\ell'}(K) - 1. \label{eq.mubar}
\end{align}
Then using $\mu_1'(H) = \mu_2'(H)$ and $\mu_1'(K) = \mu_2'(K)$, we have $\overline{\mu_1'}(H\cap K) =\overline{\mu_2'}(H\cap K)$ for every $K$. 

\textit{Step 3. Construct a partitioning hemisphere $K_t$ of $\mathcal{S}^n$ such that $H\cap H_t = H\cap K_t$ for every $t \in \mathcal{S}^n$.}

For any $t \in \mathcal{S}^n$, we can construct a partitioning hemisphere $K_t$ of $\mathcal{S}^n$ such that $H\cap H_t = H\cap K_t$. To see this, let $\Omega \coloneqq \partial H \cap \partial H_t$. Since 
\[\partial H = \{x \in \mathcal{S}^n \, \vert \, x \cdot u = 0\}\]
for some $u \in \mathcal{S}^n$, we have
\[\Omega = \{x \in \mathcal{S}^n \, \vert \, x \cdot u = 0, x \cdot t = 0\}.\]
Then for any $x \in \Omega$, we have $-x \in \Omega$; hence $\Omega$ consists of antipodal pairs. Choosing a Borel set $F \subset \Omega$ that contains exactly one point from each antipodal pair of $\Omega$, we have
\begin{equation}\label{eq.FOmega}
F \cup R(F) = \Omega \quad \text{and} \quad F \cap R(F) = \emptyset.
\end{equation}
Now let
\[K_t = H_t \cup (\partial H_t \cap R(H)) \cup F.\]
Noting that
\[R(K_t) = H_{-t} \cup (\partial H_t \cap H) \cup R(F),\]
and in view of \eqref{eq.FOmega}, it is easy to verify that $K_t$ is a partitioning hemisphere of $\mathcal{S}^n$. Also, $K_t$ satisfies $H\cap K_t = H\cap H_t$.

Therefore, result from step 2 entails $\overline{\mu_1'}(H\cap H_t) =\overline{\mu_2'}(H\cap H_t)$ for every $t \in \mathcal{S}^n$. Then it follows from Lemma~\ref{lem.injective}(i) that $\overline{\mu_1'} = \overline{\mu_2'}$.


Turning to part (ii), Lemma 2.1 of \cite{lyons2020} states that the anti-symmetrization of a probability measure on $\mathcal S^n$ is determined by its hemisphere values on any dense set of normals. Applying this to $\mu_1'$ and $\mu_2'$ yields $\overline{\mu_1'}=\overline{\mu_2'}$.

\end{proof}

\begin{proof}[\normalfont\bfseries Proof of Theorem~\ref{Thm.Gd2}] Part (i).
Let $\mu_1,\mu_2\in\mathcal{P}_1(G)$ be two probability measures concentrated on $G$, and suppose that $E_d(\nu)=0$ for $\nu=\mu_1-\mu_2$.
By \eqref{eq.equality-data} and Lemma~\ref{lem.volume-positive}, the hemisphere
masses of $\mu'_\ell=P_\#\mu_\ell$ agree on a dense set of normals. Then Lemma~\ref{lem:antisymm} gives $\overline{\mu'_1}=\overline{\mu'_2}.$

For a positive measure $\gamma$, its maximal invariant submeasure is
$\gamma_R=\gamma\wedge R_\#\gamma$, i.e., the greatest lower bound of
the two measures. It is invariant under $R$. Relative to the
dominating measure $\gamma+R_\#\gamma$, its density is the minimum
of the two densities. Thus $\gamma-\gamma_R$ and its reflected
measure are mutually singular, and the Jordan decomposition gives
\[
\gamma-\gamma_R=(\gamma-R_\#\gamma)^+=\overline\gamma^{\,+}.
\]
Here is the density calculation underlying this decomposition.
Let $\rho=\gamma+R_\#\gamma$, and write
$f=d\gamma/d\rho$ and $g=d(R_\#\gamma)/d\rho$. Then
\[
\frac{d\gamma_R}{d\rho}=\min(f,g),\qquad
\frac{d(\gamma-\gamma_R)}{d\rho}=(f-g)_+.
\]
The reflected remainder has density $(g-f)_+$, so the two remainders
are supported on disjoint sets. Moreover, every invariant measure
$\tau\leq\gamma$ also satisfies $\tau=R_\#\tau\leq R_\#\gamma$;
its density is bounded by $\min(f,g)$. This proves the maximality
of $\gamma_R$, as well as the asserted Jordan decomposition.
It follows that
$\mu'_1-\mu'_{1,R}=\mu'_2-\mu'_{2,R}$.
Both invariant measures are concentrated on
$P(G)\cap R(P(G))$. If there is no antipodal pair, both vanish.
If the unique pair is $\{z,-z\}$, each has the form
$c_\ell(\delta_z+\delta_{-z})$. The equality of the non-invariant
parts and the total probability masses implies $2c_1=2c_2$.
In detail, let $q=\mu'_1-\mu'_{1,R}=\mu'_2-\mu'_{2,R}$.
Taking total masses gives
\[
1=q(\mathcal S^n)+\mu'_{1,R}(\mathcal S^n)
 =q(\mathcal S^n)+\mu'_{2,R}(\mathcal S^n).
\]
An invariant measure supported on $\{z,-z\}$ assigns equal mass
to these two points, so equality of total masses determines that
measure completely. Consequently $\mu'_{1,R}=\mu'_{2,R}$.
Hence $\mu'_1=\mu'_2$, and radial injectivity on $G$ gives
$\mu_1=\mu_2$. The converse is immediate.

Part (ii). Suppose $(G, d)$ is of SNT, and assume that either (INJ) or (ODA) fails. Then Remark~\ref{rmk.ThmGd2} implies that when $J \equiv 0$, $(G, d)$ is not of SNT, which is a contradiction. Hence,  (INJ) and (ODA) are also necessary.
\end{proof}

\begin{proof}[\normalfont\bfseries Proof of Theorem~\ref{thm.typeIandII}] First, we consider the case that $\mathcal X$ is almost $k_p$-star-shaped. Since $\mathcal X_0\setminus\mathcal X$ has zero induced
$n$-dimensional measure in $\mathcal X_0$, the pushforward measures satisfy $P_\#(\sigma_{\mathcal X_0}|_\mathcal X)=P_\#\sigma_{\mathcal X_0}$.
Reindexing remaining points on a ray does not change this aggregate measure. As the same penalty is restricted from $\mathcal X_0$, one has
$d_\mathcal X(x,y)=d_{\mathcal X_0}(x,y)$ for $x,y\in\mathcal X$. NT follows
by restriction. The hypotheses on $G$ also hold in $\mathcal X_0$, so
Theorem~\ref{Thm.Gd2} gives SNT of $(G,d_{\mathcal X_0})$ under the conditions (INJ) and (ODA). The same conclusion holds for $(G,d_{\mathcal X})$, since $d_\mathcal X =d_{\mathcal X_0}$ on $G \times G$.

Next, consider the case that $\mathcal X$ is essentially $k_p$-star-shaped.
Note that the area formula \eqref{eq.area-formula} applies to
each layer. For the null set $N \subseteq \mathcal{S}^n$ in Definition~\ref{def.essential}, since $\mathcal S^n\setminus N$ has full measure, Lemma~\ref{lem.volume-positive} implies $\sigma_{\mathcal S^n}\ll m$.
For completeness, write $\sigma_i=\sigma_{\mathcal X}|_{L_p(i)}$
and $P_i=P|_{L_p(i)}$. Since each $P_i$ is injective, the area formula gives
\[
\sigma_{\mathcal S^n}(A)=\int_{P_i^{-1}(A)}\operatorname{Jac}_nP(x)\,{\rm d}\sigma_i(x),
\qquad A\subseteq P(L_p(i)).
\]
In particular, $\sigma_{\mathcal S^n}(A) > 0$ implies
$\sigma_i(P_i^{-1}(A))>0$. The absolute-continuity step can also be read directly from the
layers. If $m(B)=0$ for a Borel set $B\subseteq\mathcal S^n$, then
$\sigma_i(P_i^{-1}(B))=0$ for every $i$. The area formula implies
$\sigma_{\mathcal S^n}(B\cap P(L_p(i)))=0$ for every layer.
Their union contains $B\setminus N$, where $N$ is null, so
$\sigma_{\mathcal S^n}(B)=0$. This is exactly
$\sigma_{\mathcal S^n}\ll m$, which is needed below to pass from
zero energy to spherical almost-everywhere hemisphere data.
The full-rank hypothesis also gives the reverse absolute continuity,
although only the first direction is needed in the equality argument.

Now, we prove that $d$ is a metric of NT. For $x \ne y$ and $P(x)\ne P(y)$, put
\[
A_{x,y}=\{t\in\mathcal S^n:
|1_{H_t}(P(x))-1_{H_t}(P(y))|^2=1\}.
\]
This set has positive spherical volume, and removing the exceptional
null set $N$ does not change that fact. Since
$A_{x,y}\setminus N\subseteq\bigcup_{i=1}^kP(L_p(i))$, some layer
has a positive-volume intersection with $A_{x,y}$, and hence
$m(A_{x,y})>0$. Therefore $d(x,y)\geq m(A_{x,y})>0$. For $x \ne y$ but $P(x) = P(y)$, the separation condition \eqref{eq.fibre-separation} also entails $d(x,y) > 0$. Then the triangle inequality together with \eqref{eq.I} prove that $d$ is a metric
of NT. 

Next we prove that $(G, d)$ is of SNT under the conditions (INJ) and (ODA). If the difference of two probability measures on $G$
has zero energy, \eqref{eq.equality-data} gives zero total
hemisphere masses for spherical volume-almost every normal.
More explicitly, recall that for $\nu=\mu_1-\mu_2$, the energy satisfies
\[
E_d(\nu)=E_J(\nu)-2\sum_{i=1}^k\int_{\mathcal S^n}
\bigl[\nu(P^{-1}(H_t))\bigr]^2\,dm_i(t)\leq0.
\]
If equality holds, all nonpositive terms vanish. Absolute continuity
$\sigma_{\mathcal S^n}\ll m$ then entails
$\mu'_1(H_t)=\mu'_2(H_t)$ on a full-volume, hence dense, set of normals,
where $\mu'_\ell=P_\#\mu_\ell$, $\ell = 1, 2$.
Lemma~\ref{lem:antisymm} yields
$\overline{\mu'_1}=\overline{\mu'_2}$. The hypothesis of at most one
antipodal pair removes the remaining invariant part exactly as in
Theorem~\ref{Thm.Gd2}. Injectivity of $P|_G$ then gives
$\mu_1=\mu_2$. This proves the SNT of $(G, d)$.

Finally, for an almost or essentially star-shaped manifold, suppose $J \equiv 0$ and $(G, d)$ is of SNT. If either (INJ) or (ODA) fails, then Remark~\ref{rmk.ThmGd2} implies that $(G, d)$ is not of SNT, which is a contradiction. Hence,  (INJ) and (ODA) are also necessary.
\end{proof}



\begin{proof}[\normalfont\bfseries Proof of Theorem~\ref{thm:radial_completion}]
Neither the binary triangle inequality nor the calculation
\eqref{eq.I} uses full radial coverage. Thus $d$ is a finite
negative-type pseudometric, and \eqref{eq.general-separation}
is precisely the remaining point-separation requirement.
If $\nu=\mu_1-\mu_2$ on $G$ has zero energy $E_d(\nu)$, the two nonpositive
terms in \eqref{eq.I} vanish separately, giving
$(P_\#\nu)(H_t)=0$ for $m$-almost every $t\in P(\mathcal X)$.
By \eqref{eq.restricted-injectivity}, $P_\#\nu=0$. Injectivity of
$P|_G$ and its Borel inverse then imply $\nu=0$.

We also retain the following spherical limiting argument as an alternative proof of the NT. Set
\[
m_\varepsilon=m+\varepsilon^n\sigma_{\mathcal S^n},\qquad
d_\varepsilon(x,y)=J(x,y)+\int_{\mathcal S^n}
|1_{H_t}(P(x))-1_{H_t}(P(y))|^2\,dm_\varepsilon(t).
\]
If $0<\varepsilon<\operatorname{dist}(p,\mathcal X)$, this is the
contribution of the disjoint auxiliary sphere
$S_\varepsilon(p)=\{p+\varepsilon t:t\in\mathcal S^n\}$ to
$\mathcal X\cup S_\varepsilon(p)$, restricted to $\mathcal X$ with
the same penalty. Radial indices shift by one, but the aggregate
original directional measure is unchanged. When no disjoint sphere
is available, the definition of the measure $m_\varepsilon$ still applies. Since the binary
integrand is bounded by one, as $\varepsilon \rightarrow 0$,
\begin{equation}\label{eq.uniform-bound}
\begin{split}
\|d_\varepsilon-d\|_\infty
&:=\sup_{x,y\in\mathcal X}|d_\varepsilon(x,y)-d(x,y)|\\
&\leq\varepsilon^n\sigma_{\mathcal S^n}(\mathcal S^n)\longrightarrow0.
\end{split}
\end{equation}
Indeed, the additional directional measure contributes exactly
\[
d_\varepsilon(x,y)-d(x,y)
=\varepsilon^n\int_{\mathcal S^n}
|1_{H_t}(P(x))-1_{H_t}(P(y))|^2\,d\sigma_{\mathcal S^n}(t).
\]
The integrand lies between zero and one, independently of $x,y$.
This explains both the nonnegativity of the difference and the
uniform, rather than merely pointwise, convergence in
\eqref{eq.uniform-bound}.
Every $d_\varepsilon$ has NT by \eqref{eq.I} with
$m_\varepsilon$ in place of $m$. For a finite signed measure $\nu$
of total mass zero and integrable energy,
\[
|E_{d_\varepsilon}(\nu)-E_d(\nu)|
\leq\|d_\varepsilon-d\|_\infty\,|\nu|(\mathcal X)^2\longrightarrow0.
\]
In expanded form, the estimate follows from
\begin{align*}
\left|\iint(d_\varepsilon-d)(x,y)\,{\rm d}\nu(x) {\rm d}\nu(y)\right|
&\leq\iint|d_\varepsilon-d|(x,y)\,d|\nu|(x)d|\nu|(y)\\
&\leq\varepsilon^n\sigma_{\mathcal S^n}(\mathcal S^n)
|\nu|(\mathcal X)^2.
\end{align*}
Thus the limit may be passed through the energy integral, giving
\[
\iint d(x,y)\,{\rm d}\nu(x){\rm d}\nu(y)
=\lim_{\varepsilon\downarrow0}\iint d_\varepsilon(x,y)\,{\rm d}\nu(x){\rm d}\nu(y)
\leq0.
\]
Consequently $E_d(\nu)\leq0$. This supplies a second proof of
NT. The equality case, however, is justified by
\eqref{eq.I} and the restricted injectivity hypothesis, not by
preservation of strict inequalities under the limit.
\end{proof}

\begin{proof}[\normalfont\bfseries Proof of Proposition~\ref{prop:zero-exact}]
Recall that $d=D$ when $J=0$. The energy identity \eqref{eq.I} becomes
\[
E_d(\mu_1-\mu_2)
=-2\int_{\mathcal S^n}
\bigl[\mu_1'(H_t)-\mu_2'(H_t)\bigr]^2\,dm(t),
\]
where $\mu_\ell'=P_\#\mu_\ell$, $\ell = 1, 2$.
Thus zero energy means exactly that the two projected probabilities
have the same hemisphere values $m$-almost everywhere.
Since $m$ is finite, $d$ is bounded and all probabilities have finite
first moments.

If $(G,d)$ is a metric space of SNT, points on the same radial ray
must coincide, since their angular distance is zero. This gives
(INJ). Any two probabilities on $P(G)$ can then be lifted to $G$
using the Borel inverse of $P|_G$. If their hemisphere values agree,
the displayed identity and SNT show that the lifts, and hence the
original probabilities, agree. This proves
\eqref{eq:probability-hemispheres}.

Conversely, apply \eqref{eq:probability-hemispheres} first to
$\delta_{P(x)}$ and $\delta_{P(y)}$. If $d(x,y)=0$, their hemisphere
values agree $m$-a.e., so $P(x)=P(y)$; (INJ) gives $x=y$.
Hence $d$ is a metric on $G$. Applying the same argument to arbitrary
probabilities and using the energy identity proves SNT.

To see why (ODA) follows, note first that
\begin{equation}\label{eq:boundary-null-addition}
m(\partial H_z)=0\qquad(z\in\mathcal S^n).
\end{equation}
Indeed, its preimage is $\mathcal X\cap(p+z^\perp)$.
In the $n$-dimensional plane $p+z^\perp$, each radial ray meets this
set in at most $k$ points. Polar coordinates and Fubini show that
its $n$-dimensional volume is zero, as required.
If $P(G)$ contained two distinct antipodal pairs, the uniform
probabilities on these pairs would therefore have the same
hemisphere values $m$-a.e. This contradicts
\eqref{eq:probability-hemispheres}.

Finally, Example~\ref{ex:hsm-not-necessary} shows that (HSSM) need not follow.
\end{proof}





\begin{proof}[\normalfont\bfseries Proof of Proposition~\ref{prop:torus-positive-patch}]
We prove the result for $G_+$, and the same arguments apply also to $G_-$. The set $U$ is an open neighbourhood of $(1,0,\ldots,0)$, so $G_+$ is a smooth $n$-dimensional patch with $\mathcal H^n(G_+)>0$. The radial projection is a diffeomorphism on this patch. We prove the restricted injectivity needed by Theorem~\ref{thm:radial_completion}.

Let $\lambda$ be a finite signed measure concentrated on $U$ with
$\lambda(H_t)=0$ for $g_+\dd\sigma_{\mathcal S^n}$-almost every $t\in B_{R, r}$.
Since $g_+>0$ on the open belt, this holds for $\sigma_{\mathcal{S}^n}$-almost every such normal. Use the $n$-dimensional gnomonic chart
\[
\Gamma(z)=(z_2/z_1,\ldots,z_{n+1}/z_1),\qquad
\widetilde\lambda=\Gamma_\#\lambda.
\]
Its support is contained in the closed ball of radius
$\sqrt{1-\eta^2}/\eta$ in $\mathbb R^n$. Put
\[
t(s,\omega)=\frac{(-s,\omega)}{\sqrt{1+s^2}}, \qquad s \in \mathbb{R}, \, \omega \in \Omega,
\]
where $\Omega=\{\omega\in\mathcal S^{n-1}:|\omega_n|<a-\delta\}$.
This is a nonempty open family for $n\geq2$, and every real offset
$s$ gives a normal in $B_{R,r}^\circ$. The corresponding hemisphere
mass equals $F_\omega(s)=\widetilde\lambda\{u:\omega\cdot u>s\}$.
The coordinate volume element is
\[
\dd\sigma_{\mathcal S^n}(t(s,\omega))
=(1+s^2)^{-(n+1)/2}\dd s\dd\sigma_{\mathcal S^{n-1}}(\omega).
\]
It is positive, so Fubini gives $F_\omega=0$ almost everywhere in
$s$ for almost every $\omega\in\Omega$. In distributions,
\[
-\partial_sF_\omega=(u\mapsto\omega\cdot u)_\#\widetilde\lambda=0.
\]
Consequently $\widehat{\widetilde\lambda}(\xi\omega)=0$ for every
$\xi\in\mathbb R$ and almost every $\omega\in\Omega$. Continuity extends
this to all $\omega\in\Omega$. Compact support makes the Fourier
transform real analytic on $\mathbb R^n$. Its vanishing on the nonempty
open cone $\{\xi\omega:\xi>0,\omega\in\Omega\}$ implies that it
vanishes identically. Fourier uniqueness gives
$\widetilde\lambda=0$ and then $\lambda=0$. This proves
\eqref{eq.restricted-injectivity}, and hence $(G_+,d_Q)$ has strong negative type for every $n\geq2$. 
\end{proof}

\begin{proof}[\normalfont\bfseries Proof of Proposition~\ref{prop:product-torus}]

Since $\varepsilon>0$, $J$ in \eqref{eq.Jtorus} separates distinct points and satisfies the triangle inequality. For every nonzero difference $\nu$ of two probability measures, additivity of energy gives
\[
E_{d_{T,\varepsilon}}(\nu)
=E_{D_T}(\nu)+\varepsilon E_{\|X(\cdot)-X(\cdot)\|}(\nu)<0,
\]
where the first term is nonpositive due to \eqref{eq.ht}, and the second term is strictly negative because Euclidean space is of SNT \cite[Theorem~3.16]{lyons2013} and $X$ is injective. Also, all measures have finite first moments since the torus and metric are bounded. 

Moreover, the metric $d_{T,\varepsilon}$ has the ordinary topology: its lower bound is $\varepsilon$ times chord distance, and continuity of $D_T$ follows by dominated convergence. To see the latter, fix all but one angle
$\phi_j$ in $K$; since $a_j>0$, its zero set in that angle is finite, and Fubini makes the full zero set volume-null.
\end{proof}



\providecommand{\bysame}{\leavevmode\hbox to3em{\hrulefill}\thinspace}
\providecommand{\MR}{\relax\ifhmode\unskip\space\fi MR }
\providecommand{\MRhref}[2]{%
  \href{http://www.ams.org/mathscinet-getitem?mr=#1}{#2}
}
\providecommand{\href}[2]{#2}






    



\end{document}